\documentclass[11pt,a4paper]{article}

\usepackage[naustrian, english]{babel} 
\usepackage[T1]{fontenc}
\usepackage{lmodern}
\usepackage{textcomp}
\usepackage{amssymb} 
\usepackage{amsthm}
\usepackage{mathtools}
\usepackage[pdftex]{graphicx}
\usepackage[export]{adjustbox}
\usepackage{caption}
\usepackage{color}
\usepackage[arrow, matrix, curve]{xy}
\usepackage{lmodern}
\usepackage[utf8]{inputenc}
\usepackage{CJKutf8}
\usepackage{url}
\usepackage{hyperref}
\usepackage{float}
\usepackage{chngcntr}
\usepackage{listings}
\usepackage{appendix}
\usepackage{pdfpages}
\usepackage{algorithm}
\usepackage{algorithmic}
\usepackage{geometry}
\usepackage{bbm}
\usepackage{bbold}
\usepackage{graphicx}
\usepackage{caption}
\usepackage{subcaption}
\usepackage{xcolor}
\usepackage{xargs} 
\usepackage[colorinlistoftodos,textsize=small]{todonotes}

\definecolor{codegreen}{rgb}{0,0.6,0}
\definecolor{codegray}{rgb}{0.5,0.5,0.5}
\definecolor{codepurple}{rgb}{0.58,0,0.82}

\lstdefinestyle{mystyle}{
	basicstyle=\tiny,
	commentstyle=\color{codegreen},
	keywordstyle=\color{blue},
	numberstyle=\tiny\color{codegray},
	stringstyle=\color{codepurple},
	breakatwhitespace=false,         
	breaklines=true,                 
	captionpos=b,                    
	keepspaces=true,                 
	numbers=left,                    
	numbersep=5pt,                  
	showspaces=false,                
	showstringspaces=false,
	showtabs=false,                  
	tabsize=2,
}

\newcommand{\R}{\mathbb{R}}

\newcommand{\pa}{\partial}

\newcommand{\dt}{\frac{d}{dt}}

\newcommand{\ve}{\varepsilon}

\newcommand{\vp}{\varphi}

\newcommand{\n}{\mathcal{N}}

\newcommand{\md}{\,\mathrm{d}}

\newcommand{\vpa}{\varphi_\ast}

\newcommand{\T}{\mathbb{T}^1}
\newcommand{\Tt}{\mathbb{T}^2}

\newcommand{\TAin}{\mathbb{T}^{AL}_{\to\vp}}

\newcommand{\TRout}{\mathbb{T}^{REV}_{\vp}}

\newcommand{\h}{\mathcal{H}}

\newcommand{\xia}{\xi_{\ast}}

\newcommand{\mua}{\mu^{\ast}}
\newcommand{\af}{f^{\ast}}

\newtheorem{theorem}{Theorem}
\newtheorem{lemma}[theorem]{Lemma}
\newtheorem{cor}[theorem]{Corollary}
\theoremstyle{definition}
  
\newtheorem{rem}{Remark}

\newcommandx{\unsure}[2][1=]{\todo[linecolor=red,backgroundcolor=red!25,bordercolor=red,#1]{#2}}
\newcommandx{\change}[2][1=]{\todo[linecolor=blue,backgroundcolor=blue!25,bordercolor=blue,#1]{#2}}
\newcommandx{\info}[2][1=]{\todo[linecolor=green,backgroundcolor=green!25,bordercolor=green,#1]{#2}}
\newcommandx{\improvement}[2][1=]{\todo[linecolor=yellow,backgroundcolor=yellow!25,bordercolor=yellow,#1]{#2}}

\newcommandx{\biblio}[2][1=]{\todo[linecolor=blue,backgroundcolor=magenta!25,bordercolor=blue,#1]{#2}}
\newcommandx{\laura}[2][1=]{\todo[linecolor=violet,backgroundcolor=violet!25,bordercolor=violet,#1]{#2}}
\newcommandx{\michi}[2][1=]{\todo[linecolor=purple,backgroundcolor=red!40,bordercolor=purple,#1]{#2}}

\begin{document}
	\title{Kinetic Model for Myxobacteria with Directional Diffusion}
	\author{L. Kanzler\thanks{Sorbonne Université, Laboratoire Jacques-Louis Lions, F-75005 Paris, France.
			{\tt laura.kanzler@sorbonne-universite.fr}} \and 
		C. Schmeiser\thanks{University of Vienna, Faculty for Mathematics, Oskar-Morgenstern-Platz 1, 1090 Wien, Austria. 
			{\tt christian.schmeiser@univie.ac.at}}}
	\date{\vspace{-5ex}}
	\maketitle
	
	\begin{abstract}
	In this article a  kinetic model for the dynamics of myxobacteria colonies on flat surfaces is investigated. The model is based on the kinetic equation for collective bacteria dynamics introduced in [S. Hittmeir, L. Kanzler, A. Manhart, C. Schmeiser, KRM, 14 (1), pp. 1--24, 2021], which is based on the assumption of hard binary collisions of two different types: alignment and reversal, but extended by additional Brownian forcing in the free flight phase of single bacteria. This results in a diffusion term in velocity direction at the level of the kinetic equation, which opposes the concentrating effect of the alignment operator. A global existence and uniqueness result as well as exponential decay to uniform equilibrium is proved in the case where the diffusion is large enough compared to the total bacteria mass. Further, the question wether in a small diffusion regime nonuniform stable equilibria exist is positively answered by performing a formal bifurcation analysis, which revealed the occurrence of a pitchfork bifurcation. These results are illustrated by numerical simulations.
	\end{abstract}
	
	\medskip

	
 {\it Acknowledgements:} This work has been supported by the Austrian Science Fund, grants no. W1245 and F65. 
				
	\section{Introduction}
	The aim of this work is to investigate a model for the dynamics of myxobacteria colonies moving on flat substrates. 	
	The equation of interest is the kinetic transport equation 
	\begin{equation}\label{model}
	\pa_t f + \omega(\vp) \cdot \nabla_x f = \mu \pa_{\vp}^2 f+ Q(f,f).
	\end{equation}	
	for the distribution function $f(x,\vp,t)\geq0$, where $x \in \Tt, \vp \in \T$ and $t \geq0$ denote position, the directional angle, and time, respectively. We consider the collision operator $Q$ introduced in \cite{ourpaper} and extend the model by a diffusion term with diffusivity $\mu>0$ in the angular direction.
Under the assumption of constant speed (normalized to 1), the velocity is given by $\omega(\vp)=(\cos\vp,\sin\vp)$. The notation $\T$ and $\Tt$ is used for the one- and,
respectively, two-dimensional flat tori with $2\pi$-periodicity. The collision operator is of the form
	\begin{align}\label{Q}
	Q(f,g)=2\int_{\TAin}  b(\tilde\vp,\vpa) \tilde f g_* d\vpa  
	+ \int_{\TRout} b(\vp^\downarrow,\vpa^\downarrow) f^\downarrow g_*^\downarrow d\vpa - \int_{\T} b(\vp,\vpa) f g_* 
	d\vpa \,,
	\end{align}
	where
	$$
	\TRout := \left(\vp+\frac{\pi}{2},\vp+\frac{3\pi}{2}\right) \,,\qquad \TAin = \left( \vp-\frac{\pi}{4}, \vp+\frac{\pi}{4}\right) \,,
	$$ 
	and
	$$
	\tilde\vp := 2\vp-\vpa \,,\qquad \vp^\downarrow = \vp+\pi \,,\qquad \vpa^\downarrow = \vpa+\pi\,.
	$$
Super- and subscripts on $f$ and $g$ denote evaluation at $\vp$ with the same super- and subscripts. The model describes movement along trajectories governed by \emph{Brownian motion} in velocity direction, interrupted by hard binary collisions with \emph{collision cross-section} $b(\vp,\vpa)$. Its dependence on the pre-collisional
directions $\vp$ and $\vpa$ is due to the shape of the bacteria. In this work we consider two possible choices: 
	Rod shaped bacteria are described by $b(\vp,\vpa) = |\omega_*\cdot\omega^\bot| = |\sin(\vp-\vpa)|$. On the other hand, bacteria with circular shape yield a collision rate independent from the pre-collisional directions. By analogy to a similar simplification of the gas dynamics Boltzmann equation \cite{Cercignani}, we use the name \textit{'Maxwellian myxos'} for this imagined species, modeled by $b(\vp,\vpa)\equiv 1$. We may note at this point the reflection and rotation symmetries 
$b(\vp, \vpa) = b(\vp^\downarrow, \vpa^\downarrow)$ and, respectively, $b(\vp+\alpha,\vpa+\alpha)= b(\vp,\vpa)$.
The gain terms in \eqref{Q} describe two different types of collisions:
	\begin{itemize}
		\item \textit{Alignment:} $(\tilde\vp,\vpa)\to(\vp,\vp)$ with $\vp = (\tilde\vp+\vpa)/2$, if two myxobacteria moving in directions $\tilde{\vp}$ and $\vpa$ meet at an angle smaller than $\pi/2$. The factor 2 is due to the fact that an alignment collision produces 2 myxobacteria with the same
		direction. The set $\TAin$ describes all angles $\vpa$, which can produce the angle $\vp$ upon collision.
		\item \textit{Reversal:} $(\vp,\vpa)\to(\vp^\downarrow,\vpa^\downarrow)$, if two myxobacteria with directions $\vp$ and $\vpa$ meet at an angle larger than $\pi/2$. The set $\TRout$ describes all angles $\vpa$ such that a collision between $\vp^\downarrow$ and $\vpa^\downarrow$ can produce the angle $\vp$.
	\end{itemize}
Properties of the model without directional diffusion
	\begin{align}\label{myxo}
	\pa_t f + \omega(\vp) \cdot \nabla_x f = Q(f,f) \,,
	\end{align}
introduced and investigated in \cite{ourpaper}, will serve as motivation for the dynamics we expect in \eqref{model} in the small diffusion regime. In both \eqref{model} and \eqref{myxo} the total mass is conserved and denoted by
	$$M:=\int_{\T\times \Tt}f(x,\vp,t) \: \md \vp \md x.$$
	Throughout all of this paper, we denote the \emph{uniform distribution} by
	\begin{align}\label{f0}
	f_0:= \frac{M}{2\pi},
	\end{align}
which defines an equilibrium for both \eqref{model} and \eqref{myxo}. Numerical experiments in the non-diffusive case \cite{ourpaper} suggest instability of $f_0$
and convergence as $t\to\infty$ to an equilibrium measure of the form
	\begin{align}\label{measureequ}
	f_\infty(\vp) :=\rho_+ \delta(\vp-\vp_+) +\rho_-\delta(\vp-\vp_+^\downarrow) \,.
	\end{align}
The convergence can be proved for the spatially homogeneous equation with special initial conditions. These observations give rise to the assumption that for small $\mu$ the uniform equilibrium $f_0$ will also be unstable for \eqref{model} and other equilibria will occur. 
	
	In \cite{BDG1} a similar model, also of Boltzmann-type but just describing alignment interactions, was introduced as binary collision counterpart of the Vicsek model for swarm dynamics, which on the other hand is based on nonlocal alignment interactions between agents \cite{VCBCS}. It was investigated further in \cite{BDG2} as well as in \cite{carlen}, where additionally the case of Brownian forcing between binary interactions was considered. Before, such a diffusive kinetic equation modelling alignment between agents was already introduced and studied in \cite{BK}.
	
Section \ref{sec:ex} of this article is dedicated to proving global existence of a solution subject to initial conditions sufficiently close the uniform equilibrium $f_0$ as well as exponential decay towards this steady state, both under the assumption that the ratio $\mu/M$ is large enough. This result relies on a perturbative approach including the proof of spectral stability of the equilibrium before extending it to the nonlinear framework, close to equilibrium. We want to mention at this point that the theory for the dissipative Boltzmann equation is much less developed than the one of the conservative Boltzmann equation, which is due to the lack of a-priory estimates given by an entropy. Global existence results for the spatially inhomogeneous Cauchy problem in the inelastic case are only known for near vacuum data \cite{alonso} (i.e. the collisions do not have much impact on the dynamics) inspired by the method using Kaniel \& Shinbrot iterates \cite{KS}. More recently in \cite{T} existence in the spatially inhomogeneous framework for inelastic collisions could be established without the closeness to vacuum restriction. Further, theory in the one-dimensional situation, where grazing collisions are almost elastic, can be found in \cite{BP}. Another important work in the one dimensional case has been done in \cite{jabin}, carrying out the rigorous macroscopic limit towards pressureless gas dynamics. Many more results have been established in an homogeneous framework, see e.g. \cite{GPV} and \cite{MMR} for investigations on the existence and uniqueness of solutions. Besides work on the Cauchy problem, a number of results on existence and further properties of self-similar profiles for diffusively excited inelastic hard sphere models have been obtained. Among them to mention \cite{BGP}, \cite{GPV} and \cite{MM1, MM2, MM3} for the case of a constant coefficient of restitution, while we refer the reader to \cite{AL} for considerations on the non-constant case. 
	
	
	In the first part of Section \ref{sec:bif} we investigate stability of the uniform steady state $f_0$ and the existence of nontrivial spatially homogeneous equilibria in dependence of the diffusivity and total mass of the system, using bifurcation theory \cite{cran} via Fourier series expansion. We establish the occurrence of a supercritical pitchfork bifurcation. Although the calculations remain formal, it does provide new insights into the behavior of the model, while being consistent with already existing results. Indeed, in \cite{carlen} a rigorous proof of existence of a pitchfork bifurcation in the noisy version of a Boltzmann-type alignment model for swarming behavior is stated. It is interesting to note that the set of nontrivial equilibria is two-dimensional with two opposite peaks of equal height, as opposed to the non-diffusive case
\eqref{measureequ} where the set of unsymmetric equilibria is three-dimensional. In
the second part of Section \ref{sec:bif} we investigate formal approximations of equilibria for small values of the diffusivity $\mu$. For the case of Maxwellian myxos
we prove the existence of a formal approximation of equilibria, which can be seen as regularizations of \eqref{measureequ} in the case $\rho_+=\rho_-$.
	
In Section \ref{sec:num} we present results of numerical simulations for the spatially homogeneous equation, providing evidence for the bifurcation results of 
Section \ref{sec:bif}.
	
	\section{Decay to the uniform equilibrium}\label{sec:ex}
	
	The aim of this section is to establish existence and uniqueness of solutions of \eqref{model}, as well as asymptotic stability of the uniform equilibrium. 
	This can only be expected under the assumption of large enough diffusivity $\mu$ compared to the total mass $M$. 
	
	\begin{theorem}\label{existence}
		Let $f_I \in H_{x,\vp}^2(\Tt\times\T)$, $f_I\ge 0$, and let $\mu/M$ be large enough with 
		$M=\int_{\Tt\times\T}f_I \md\vp\md x$. Let furthermore $\|f_I-f_0\|_{H_{x,\vp}^2(\Tt\times\T)}$ be small enough with $f_0 = M/(2\pi)$.
		Then equation \eqref{model} subject to the initial condition $f(t=0)=f_I$ has a unique global solution $f\in C([0,\infty), H_{x,\vp}^2(\Tt\times\T))$,
		satisfying
		$$
		\|f(t)-f_0\|_{H_{x,\vp}^2(\Tt\times\T)} \le C e^{-\lambda t}\|f_I-f_0\|_{H_{x,\vp}^2(\Tt\times\T)} \,,\qquad C,\lambda>0 \,.
		$$
	\end{theorem}
	
	The rest of this section is dedicated to the proof of Theorem \ref{existence}. The first step will be a proof of spectral stability
	by an application of the \textit{$L^2$-hypocoercivity method} of \cite{dms}. Then this result will be extended to an $H^2$-setting in order to be able to 
	control the quadratic nonlinearities of the collision operator.

	\subsection{Spectral stability by hypocoercivity}\label{seclin}
	
	Following the notation of \cite{dms}, we write the linearization of \eqref{model} around $f_0=M/(2\pi)$ in the abstract form
	\begin{equation}\label{lin1}
	\pa_t f + Tf = Lf + Q_M f \,,
	\end{equation}
	with the dissipative operator $L:= \mu \pa_{\vp}^2$, the conservative transport operator $T:= \omega(\vp) \cdot \nabla_x$, and the linearized
	collision operator $Q_M f := Q(f_0,f) + Q(f,f_0)$, treated as a perturbation. The linear operators $T$, $L$, and $Q_M$ are closed on the Hilbert space 
	$$
	\h:=\left\{ f\in L^2(\Tt\times\T):\, \int_{\Tt\times\T} f \md\vp\md x = 0 \right\}\,, 
	$$
	and $L+Q_M-T$ generates the strongly continuous semigroup $e^{(L+Q_M-T)t}$ on $\h$. The scalar
	product and the norm on $\h$ will be denoted by $\langle\cdot,\cdot\rangle$ and, respectively, $\|\cdot\|$. The orthogonal projection to the nullspace 
	$\n(L)$ of $L$ is given by the average with respect to the angle:
	$$
	\Pi f := \frac{1}{2\pi} \int_{\T} f \md \vp \,.
	$$
	
	The decay to equilibrium relies on two coercivity properties:
	
	\paragraph{Microscopic coercivity:}
	\begin{equation}\label{mico}
	-\langle Lf, f \rangle = \mu \int_{\Tt\times\T} (\pa_{\vp} f)^2 \md \vp \md x \geq  \mu\|f-\Pi f\|^2 \,,
	\end{equation}
	where the last inequality is the Poincar\'e inequality on $\T$ with optimal Poincar\'e constant $1$.
	
	\paragraph{Macroscopic coercivity:}	
	\begin{equation}\label{maco}
	\|T \Pi f \|^2 =  \pi \int_{\Tt} |\nabla_x \Pi f|^2 \md x \ge 8\pi^3 \int_{\Tt} (\Pi f)^2 \md x  = 4 \pi^2 \| \Pi f \|^2 \,,
	\end{equation}
	where now the Poincar\'e inequality on $\Tt$ with optimal Poincar\'e constant $8\pi^2$ has been used. The macroscopic
	coercivity constant $4\pi^2$ can be seen as a lower bound for the spectrum of the symmetric operator $(T\Pi)^* T\Pi$ on $\n(L)$.
	
	\paragraph{Diffusive macroscopic limit:} The method of \cite{dms} relies on an algebraic property, which guarantees that the macroscopic limit,
	when the dissipative operator $L$ dominates the transport operator $T$, is diffusive:
	\begin{equation}\label{PiTPi}
	\Pi T \Pi  = 0 \,.
	\end{equation}
	It is easily verified in our situation. The macroscopic limit of \eqref{lin1} without the perturbation ($Q_M=0$) is the heat equation on $\Tt$. 
	
	\paragraph{The modified entropy:} A natural entropy for the unperturbed version of \eqref{lin1} is given by the square of the norm:
	$$
	\dt \frac{\|f\|^2}{2} = \langle Lf, f \rangle + \langle Q_M f,f \rangle  \,.
	$$
	The semidefiniteness of the dissipation $\langle Lf, f \rangle$, which vanishes on $\n(L)$, can be remedied by introducing the modified entropy 
	(see \cite{dms})
	\begin{align}\label{h}
	H[f]:= \frac{1}{2} \|f\|^2 + \ve \langle Af,f \rangle \,,
	\end{align}
	with an appropriately chosen small parameter $\ve > 0$, with the operator 
	\begin{align}\label{A}
	A = (1+(T\Pi)^*T\Pi)^{-1}(T\Pi)^* \,.
	\end{align}
	It has been shown in \cite[Lemma 1]{dms} that under the assumption \eqref{PiTPi}, $A$ and $TA$ are bounded operators with
	\begin{equation}\label{Abound}
	\|Af\| \le \frac{1}{2} \|(1-\Pi)f\| \,,\qquad \| TAf\| \le \|(1-\Pi)f\| \,.
	\end{equation}
	For $\ve<1$, the bound on $A$ implies the equivalence inequalities
	\begin{equation}\label{equiv}
	\frac{1-\ve}{2} \|f\|^2 \le H[f] \le \frac{1+\ve}{2} \|f\|^2 \,.
	\end{equation}
	The time derivative of the modified entropy is written as
	\begin{align}\label{dh}
	\dt H[f] = -D[f] \,,
	\end{align}
	where the dissipation is given by 
	\begin{align}
	D[f] &:= -\langle Lf, f \rangle + \ve \langle AT\Pi f, f \rangle + \ve \langle AT(1-\Pi)f,f\rangle - \ve \langle ALf,f \rangle - \ve \langle TAf,f \rangle \nonumber\\
	& \quad\,\, - \langle Q_M f,f \rangle - \ve \langle AQ_M f,f \rangle\,.\label{d}
	\end{align}
	We want to note here that the terms $-\ve \langle Af,Lf\rangle$ and $-\ve \langle Af,Q_M f \rangle$ are not represented in the formulation of $D[f]$, since they vanish due to the easily checked properties $A=\Pi A$ as well as
	\begin{equation}\label{QL-prop}
	Q_M = (1-\Pi)Q_M(1-\Pi)\,,
	\end{equation} 
	which the linearized collision operator 
	\begin{eqnarray*}
		Q_M f &=& 2f_0\int_{\TAin}  b(\tilde\vp,\vpa) (\tilde f + f_*) \md\vpa  
		+ f_0\int_{\TRout} b(\vp^\downarrow,\vpa^\downarrow) (f^\downarrow + f_*^\downarrow) \md\vpa \\
		&& - f_0\int_{\T} b(\vp,\vpa) (f + f_*) \md\vpa
	\end{eqnarray*}
	inherits from $Q$ due to mass conservation.
	Coercivity is provided by the first two terms as a combination of microscopic and macroscopic coercivity and of the observation that
	$AT\Pi$ can be interpreted as the application of the map $z\mapsto z/(1+z)$ to the operator $(T\Pi)^*T\Pi$:
	\begin{equation}\label{est1}
	-\langle Lf, f \rangle + \ve \langle AT\Pi f, f \rangle \ge \mu \|(1-\Pi)f\|^2 + \ve\frac{4\pi^2}{1+4\pi^2} \|\Pi f\|^2 \,.
	\end{equation}
	It remains to show that the last five terms in \eqref{d} can be controlled by the first two. We start with the last term of the first line. The property 
	$A=\Pi A$ and \eqref{PiTPi} imply $TA = (1-\Pi)TA$ and therefore, with \eqref{Abound},
	\begin{equation}\label{est2}
	|\langle TAf,f \rangle| = |\langle TAf,(1-\Pi)f \rangle| \le \|(1-\Pi)f\|^2 \,.
	\end{equation}
	The operator $AT$ is bounded if and only if its adjoint is bounded which, after using the self-adjointness of $\Pi$ and the skew-symmetry of $T$, can be written as 
	$$
	(AT)^*  = -T^2\Pi[1+(T\Pi)^*(T\Pi)]^{-1} \,.
	$$	
	Let us define $g:=[1+(T\Pi)^*(T\Pi)]^{-1}f$, giving 
	\begin{align*}
	(AT)^*f = -T^2\Pi g \,.
	\end{align*}
	Furthermore, the definition of $g$ is equivalent to $g-\Pi (v \cdot \nabla_x(v \cdot \nabla_x \Pi g))= f$. After applying $\Pi$ on both sides and using the 
	notation $\rho_g:=\Pi g$ and $\rho_f:=\Pi f$, the equation reads
	\begin{align*}
	\rho_g - \frac{1}{2} \Delta_x \rho_g = \rho_f \,.
	\end{align*}
	Testing against $\Delta_x \rho_g$ implies $\|\Delta_x \rho_g\|_{L_x^2} \le 2 \|\rho_f\|_{L_x^2}$. Therefore
	\begin{align*}
	\|(AT)^*f\|^2 &= \|T^2 \rho_g\|^2 \le \pi \|\nabla_x^2 \rho_g\|_{L_x^2}^2 = \pi \|\Delta_x \rho_g\|_{L_x^2}^2 \le 4\pi \|\rho_f\|_{L_x^2}^2 = 2 \|\Pi f\|^2 \,,
	\end{align*}	
	implying
	\begin{equation}\label{est3}
	|\langle AT(1-\Pi)f,f\rangle| = |\langle (1-\Pi)f, (AT)^*f\rangle| \le \sqrt{2}\, \|\Pi f\|\,\|(1-\Pi) f\| \,.
	\end{equation}
	Since, by a straightforward computation, $\Pi TL = -\mu \Pi T$ we have $AL = -\mu A$ and, thus,
	\begin{equation}\label{est4}
	|\langle ALf,f\rangle| = \mu |\langle Af,\Pi f\rangle| \le \frac{\mu}{2} \, \|\Pi f\|\,\|(1-\Pi) f\| \,.
	\end{equation}
	Finally, we deal with the perturbation terms. Using $0\le b\le 1$ we easily conclude
	$$
	(Q_M f)f \le f_0|f| \left( 6\int_{\T} |f_*| d\vpa + \pi |f^\downarrow|\right) \,,
	$$
	and therefore, with \eqref{QL-prop} and with the Cauchy-Schwarz inequality,
	\begin{equation}\label{est5}
	\langle Q_M f,f\rangle  \le 13\pi f_0 \| (1-\Pi)f\|^2 = \frac{13}{2}M \| (1-\Pi)f\|^2\,.
	\end{equation}
	Similarly,
	$$
	|Q_M f| \le f_0\left( 6\int_{\T} |f_*| d\vpa + \pi |f^\downarrow| + 2\pi |f| \right) \,,
	$$
	implying
	$$
	\|Q_M f\| \le f_0 \left( 6\sqrt{2\pi}\|f\| + 3\pi\|f\|\right) = 3M\left(2\sqrt{\frac{2}{\pi}}+1\right)\|f\| \,.
	$$
	Combining this with \eqref{Abound}, $A=\Pi A$, and with \eqref{QL-prop} gives
	\begin{equation}\label{est6}
	|\langle AQ_M f,f\rangle| \le 3M\left(\sqrt{\frac{2}{\pi}}+\frac{1}{2}\right)\|\Pi f\| \|(1-\Pi)f\|\,.
	\end{equation}

	\paragraph{Hypocoercivity:} Using our results \eqref{est1}, \eqref{est2}, \eqref{est3}, \eqref{est4}, \eqref{est5}, \eqref{est6} in \eqref{dh}, \eqref{d} gives
	\begin{eqnarray*}
		\dt H[f] &\le& -\left(\mu - \frac{13}{2}M - \ve\right) \|(1-\Pi)f\|^2 - \ve\frac{8\pi^3}{1+8\pi^3} \|\Pi f\|^2 \\
		&& + \ve \left(\sqrt{2} + \frac{\mu}{2} + 3M\left(\sqrt{\frac{2}{\pi}}+\frac{1}{2}\right)\right)\|\Pi f\|\,\|(1-\Pi)f\| \,.  
	\end{eqnarray*}
	Obviously for $\mu > 13M/2$ (as requested in Theorem \ref{existence}) and for $\ve$ small enough, the right hand side is negative definite and controls 
	$\|f\|^2 = \|\Pi f\|^2 + \|(1-\Pi)f\|^2$. With \eqref{equiv} we obtain the existence of $\lambda> 0$, such that
	$$
	\dt H[f] \le - 2\lambda H[f] \,,
	$$
	and therefore exponential decay of the modified entropy and also of $\|f\|$ by another application of \eqref{equiv}. This proves spectral stability of the
	uniform equilibrium in $L^2$.
	
	\begin{theorem}\label{thm:hypoL2}
		Let $\mu/M > 13/2$. Then there exist positive constants $\lambda$ and $C$, such that for any initial datum $f_I \in \h$, we have
		\begin{equation}\label{decay}
		\|e^{t(L+Q_M-T)}f_I\| \leq C e^{- \lambda t} \|f_I\| \,, \qquad t \ge 0 \,.
		\end{equation}
	\end{theorem}
	
	This result can easily be extended to the Sobolev space $H^2(\Tt\times\T)\cap\h$, which is continuously imbedded in $L^\infty(\Tt\times\T)$ and, thus, 
	an algebra. The procedure will only be outlined in the following.
	
	Note that $f\in H^2(\Tt\times\T)\cap\h$ implies that the partial
	derivatives of $f$ lie in $\h$. Therefore Theorem \ref{thm:hypoL2} immediately carries over to the pure $x$-derivatives, since the coefficients in
	\eqref{lin1} are $x$-independent and the $x$-derivatives thus solve the same equation. If there is also differentiation with respect to $\vp$ on the other hand,
	we have to proceed recursively. The following crucial, but technical observation that the collision operator $Q$ factorizes when derived with respect to $\vp$, will be used throughout the following considerations.
	\begin{lemma}\label{dphiQ} 
		Let $h_1, h_2 \in H^2(\Tt \times \T)$, we observe that
		$$\pa_{\vp}Q(h_1,h_2) = Q(\pa_{\vp}h_1,h_2)+Q(h_1,\pa_{\vp}h_2),$$
		and hence also 
		$$\pa_{\vp}^2 Q(h_1,h_2) = Q(\pa_{\vp}^2h_1,h_2) + 2Q(\pa_{\vp}h_1,\pa_{\vp}h_2) +Q(h_1,\pa_{\vp}^2h_2)\,$$
		i.e. the collision operator $Q(h_1,h_2)$ behaves like a pointwise product with respect to the $\vp$-derivative.
	\end{lemma}
	\begin{proof}
		This property can be seen easily by rewriting the collision operator in the form
		$$
		Q(h_1,h_2) = \int_{\T} b_{AL}(\vp-\vpa)\tilde{h}_1 h_{2*} \, \md \vpa + \int_{\T} b_{REV}(\vp-\vpa) h_1^\downarrow h_{2*}^\downarrow \, \md \vpa - \int_{\T} b(\vp-\vpa) h_1 h_{2*} \, \md \vpa \, ,
		$$
		where we defined 
		$$
		b_{AL}(\vp-\vpa):= b(\tilde{\vp},\vpa) \mathbb{1}_{\TAin}(\vpa) = b(\tilde{\vp},\vpa) \mathbb{1}_{\{ \cos{(2(\cdot))}>0 \,\&\, \cos{(\cdot)}>0 \}}(\vp-\vpa) 
		$$
		and
		$$
		b_{REV}(\vp-\vpa):=b(\vp^\downarrow,\vpa^\downarrow)\mathbb{1}_{\TRout}(\vpa)= b(\vp^\downarrow,\vpa^\downarrow) \mathbb{1}_{\left\{\cos{(\cdot)}<0\right\}}(\vp-\vpa)\,,
		$$
		with the crucial observation that $b_{AL}$, $b_{REV}$ and $b$ depend on the difference $\vp-\vpa$. Deriving the collision operator with respect to $\vp$ while using integration by parts to avoid the occurrence of derivatives of the collision kernel gives the desired result.
	\end{proof}
	
	The first $\vp$-derivative, $g := \partial_\vp f$, solves the equation
	\begin{eqnarray}\label{f_phi}
	\partial_t g + (T-L)g &=& -\omega(\vp)^\bot\cdot\nabla_x f + 2f_0\int_{\TAin}  b(\tilde\vp,\vpa) (\tilde{g} + g_*) \md\vpa  \nonumber\\
	&& + f_0\int_{\TRout} b(\vp^\downarrow,\vpa^\downarrow) (g^\downarrow + g_*^\downarrow) \md\vpa 
	- f_0\int_{\T} b(\vp,\vpa) (g + g_*) \md\vpa\,, \\
	&=& -\omega(\vp)^\bot\cdot\nabla_x f + Q_M g\,, \notag
	\end{eqnarray}
	with $(\omega_1,\omega_2)^\bot = (-\omega_2,\omega_1)$. The first term on the right hand side comes from the $\vp$-dependence of the coefficient
	in the transport term, whereas the remaining terms are $\partial_\vp(Q_M f)$, which is derived by first using the above Lemma \ref{dphiQ} before linearizing around the uniform equilibrium $f_0$.
	
	By Theorem \ref{thm:hypoL2} and by the argument above, exponential decay of the first term on the right hand side is already known.
	The remaining three terms are given by $Q_M$ applied to $g$ and therefore can be treated as a perturbation of $L-T$ as in the proof of Theorem \ref{thm:hypoL2}. For this step again property \eqref{QL-prop} is important. As a consequence, exponential decay of $g$ follows from the variation-of-constants formula for $\mu$ large enough. The derivatives 
	$\nabla_x \partial_\vp f$ (with second order $x$-derivatives in the inhomogeneity) and $\partial_\vp^2 f$ (with $\nabla_x f$ and $\nabla_x \partial_\vp f$
	in the inhomogeneity) are treated analogously, proving the following result.
	
	\begin{cor}\label{cor:hypoH2}
		For $\mu/M$ large enough there exist positive constants $\lambda$ and $C$, such that for any initial datum $f_I \in H^2(\Tt\times\T)\cap\h$, we have
		\begin{equation}\label{decayH2}
		\|e^{t(L+Q_M-T)}f_I\|_{H^2(\Tt\times\T)} \leq C e^{- \lambda t} \|f_I\|_{H^2(\Tt\times\T)} \,, \qquad  t \ge 0 \,.
		\end{equation}
	\end{cor}
	
	\begin{rem}
		The exponential rate constant $\lambda$ in Corollary \ref{cor:hypoH2} has to be chosen a little smaller than in Theorem \ref{thm:hypoL2} because
		of resonance in the inhomogeneous equations like \eqref{f_phi}. Otherwise an additional factor $t^2$ would appear as a result of the two-stage
		recursion process needed for estimating $\partial_\vp^2 f$. Also the ratio $\mu/M$ might have to be larger than in Theorem \ref{thm:hypoL2}.
	\end{rem}

	\subsection{Nonlinear stability of the uniform equilibrium}
	
	This section is devoted to the proof of Theorem \ref{existence}. We introduce the perturbation 
	$$
	h := f - f_0\in H^2(\Tt\times\T) \cap \mathcal{H} \,,
	$$ 
	satisfying, with the notation introduced above,
	\begin{equation}\label{nonlin}
	\pa_t h + Th = Lh + Q_M h + Q(h,h)  \,,\qquad h(t=0) = f_I - f_0 \,,
	\end{equation}
	and consider the mild formulation
	$$
	h(t) = e^{t(L+Q_M-T)}(f_I-f_0) + \int_0^t e^{(t-s)(L+Q_M-T)} Q(h(s),h(s)) \md s \,.
	$$
	For the estimation of the semigroup, Corollary \ref{cor:hypoH2} will be used, and apart from that we need estimates of the quadratic collision operator.
	
	\begin{lemma}\label{lem:Q-est}
		Let $h_1,h_2 \in H^2(\Tt\times\T) \cap \mathcal{H}$. Then $Q(h_1,h_2) \in H^2(\Tt\times\T) \cap \mathcal{H}$ and there exists a constant $\bar Q$ such that
		$$
		\|Q(h_1,h_2)\|_{H^2(\Tt\times\T)} \le \bar Q \,\|h_1\|_{H^2(\Tt\times\T)} \|h_2\|_{H^2(\Tt\times\T)} \,.
		$$
	\end{lemma}
	
	\begin{proof}
		Because of the Sobolev inequality 
		\begin{equation}\label{Sobolev}
		\|h\|_{L^\infty(\Tt\times\T)} + \|\nabla_x h\|_{L^4(\Tt\times\T)} + \|\partial_\vp h\|_{L^4(\Tt\times\T)} \le c_S \|h\|_{H^2(\Tt\times\T)} \,,
		\end{equation}
		it will be sufficient to find estimates in terms of the $L^\infty$-norms of $h_1$ and $h_2$ or of the $L^4$-norms of the first order derivatives. 
		We start with the observation
		$$
		|Q(h_1,h_2)| \le 4 \|h_1\|_{L^\infty(\Tt\times\T)} \int_{\T} |h_2| \md\vp \,,
		$$
		implying 
		\begin{equation}\label{Q-est1}
		\|Q(h_1,h_2)\| \le 8\pi \,\|h_1\|_{L^\infty(\Tt\times\T)} \|h_2\| \,,
		\end{equation}
		and similarly,
		\begin{equation}\label{Q-est2}
		\|Q(h_1,h_2)\| \le \sqrt{8\pi+33\pi^2} \,\|h_1\| \|h_2\|_{L^\infty(\Tt\times\T)} \,,
		\end{equation}
		Alternatively it is, by the convolution structure of the collision terms, straightforward to show
		$$
		\int_{\T} Q(h_1,h_2)^2 \md\vp \le 6\pi \int_{\T} h_1^2\md\vp  \int_{\T} h_2^2\md\vp \,,
		$$
		with the consequence
		\begin{equation}\label{Q-est3}
		\|Q(h_1,h_2)\| \le \sqrt{12\pi} \|h_1\|_{L^4(\Tt\times\T)} \|h_2\|_{L^4(\Tt\times\T)} \,.
		\end{equation}
		By elliptic regularity, we may use the equivalent norm $\|h\|_* = \|h\| + \|\Delta_x h\| + \|\partial_\vp^2 h\|$ on $H^2(\Tt\times\T)$. We have
		\begin{eqnarray*}
			\|\Delta_x Q(h_1,h_2)\| &\le& \|Q(\Delta_x h_1,h_2)\| + \|Q(h_1,\Delta_x h_2)\| + 2\|Q(\nabla_x h_1,\nabla_x h_2)\|  \\
			&\le&  \sqrt{8\pi+33\pi^2}  \,\|\Delta_x h_1\| \|h_2\|_{L^\infty(\Tt\times\T)} + 8\pi \,\|h_1\|_{L^\infty(\Tt\times\T)} \|\Delta_x h_2\| \\
			&& + \sqrt{12\pi} \|\nabla_x h_1\|_{L^4(\Tt\times\T)} \|\nabla_x h_2\|_{L^4(\Tt\times\T)} \\
			&\le& c_S(\sqrt{8\pi+33\pi^2} + 8\pi + c_S\sqrt{12\pi}) \|h_1\|_{H^2(\Tt\times\T)} \|h_2\|_{H^2(\Tt\times\T)} \,,
		\end{eqnarray*}
		where we have used \eqref{Q-est1}--\eqref{Q-est3} as well as \eqref{Sobolev}. It remains to estimate $\partial_\vp^2 Q(h_1,h_2)$.
		Due to the property stated in Lemma \ref{dphiQ} we can estimate analogously to the above,
		$$
		\|\partial_\vp^2 Q(h_1,h_2)\| \le c_S(\sqrt{8\pi+33\pi^2} + 8\pi + c_S\sqrt{12\pi}) \|h_1\|_{H^2(\Tt\times\T)} \|h_2\|_{H^2(\Tt\times\T)}    \,,
		$$
		completing the proof.
	\end{proof}
	
	Lemma \ref{lem:Q-est} implies local Lipschitz continuity of $Q$, considered as a map on $H^2(\Tt\times\T) \cap \mathcal{H}$ and therefore local
	existence and uniqueness of a mild solution. Corollary \ref{cor:hypoH2} and Lemma \ref{lem:Q-est} imply
	$$
	\|h(t)\|_{H^2(\Tt\times\T)} \le Ce^{-\lambda t}\|f_I - f_0\|_{H^2(\Tt\times\T)} + C\bar Q \int_0^t e^{\lambda(s-t)}\|h(s)\|_{H^2(\Tt\times\T)}^2 \md s \,.
	$$
	It is easily checked that for 
	$$
	\|f_I - f_0\|_{H^2(\Tt\times\T)} \le \frac{\lambda}{4C^2 \bar Q} \,,
	$$
	Picard iteration preserves the inequality
	$$
	\|h(t)\|_{H^2(\Tt\times\T)} \le 2Ce^{-\lambda t}\|f_I - f_0\|_{H^2(\Tt\times\T)} \,,
	$$
	completing the proof of Theorem \ref{existence}.
	
\section{Spatially Homogeneous Equilibria}\label{sec:bif}

This section focuses on finding nonuniform, \emph{spatially homogeneous} equilibria of (\ref{model}), i.e. stationary solutions of the equation
	\begin{equation}\label{spathom}
		\begin{split}
			&\pa_t f=\mu \pa_{\vp}^2f + Q(f,f), \quad \vp \in \T, \: t>0,\\
			&f(\vp,0)=f_I(\vp),\quad \vp \in \T,
		\end{split}
	\end{equation}
and further investigate their stability. We expect the uniform equilibrium 
$$
   f_0 = \frac{M}{2\pi} \,,\qquad\mbox{with } M = \int_{\T} f_I(\vp)d\vp \,,
$$  
to be  stable for sufficiently large diffusion and intend to find other equilibria in the collision dominated regime. 
	 
We approach this problem in two ways. On the one hand, in Section \ref{bif} a formal bifurcation analysis with bifurcation parameter $\mu$ shows a 
supercritical pitchfork bifurcation away from the uniform equlibrium, producing a branch of nontrivial equilibria for $\mu$ less than a critical value $\mua$.
The nontrivial equilibria have two reflection symmetries with two opposite maxima and two opposite minima.
We restrict ourselves to the formal computations, noting that they can be made rigorous in a straightforward way, following the theory of bifurcations from a simple eigenvalue (see, e.g., \cite{cran}). 

On the other hand, we investigate the case $\mu \ll 1$. The picture we have in mind is that, as $\mu\to 0+$, the nontrivial equilibrium converges to
$$
   f_\infty(\vp) = \frac{M}{2}(\delta(\vp) - \delta(\vp-\pi)) \,,
$$
or to a rotated version. This is motivated by the fact that $f_\infty$ is an equilibrium for $\mu=0$ \cite{ourpaper}. In Section \ref{ex} we construct a formal
approximation for a nontrivial equilibrium, which is a smoothed version of $f_\infty$.

\subsection{Bifurcation from the Uniform Equilibrium}\label{bif}
	
\paragraph{Stability of the uniform equilibrium:} We start by analyzing the spectral stability of $f_0$ by linearization of (\ref{spathom}): 
	\begin{equation}\label{myxolin}
	\pa_t \af = (L + Q_M)\af \,, \qquad \int_{\T} \af \;\md \vp =0 \,,
	\end{equation}
where $\af$ is the perturbation, and we recall $L = \mu \pa_\vp^2$ and the linearization $Q_M\af = Q(f_0,\af)+Q(\af,f_0)$ of $Q$ around $f_0$. 
For the collision kernel, the model $b(\vp,\vpa)= |\sin(\vp-\vpa)|$ for rod-shaped bacteria will be used in this section (see, however, the remark at 
the end of the section). The Fourier series expansion
	\begin{align}\label{fourierseries}
	\af(\vp,t) &=\sum_{n=1}^\infty a_n(t) \cos{(n\vp)} +\sum_{n=1}^\infty b_n(t) \sin{(n\vp)} 
	\end{align}
diagonalizes the problem and leads to
	\begin{align*}
		\dot a_n = \lambda_n a_n \,,\qquad \dot b_n = \lambda_n b_n \,,\qquad n\ge 1 \,,
	\end{align*}
with the eigenvalues 
	$$\lambda_1 = -\mu - \frac{f_0}{3}\left(4\sqrt{2} - 1\right) \,,$$
	$$\lambda_2 = -4\mu + \frac{2f_0}{3} \,,$$
	$$\lambda_n = -n^2\mu + 2f_0\left(\frac{4n\sin{(n\pi/4)}-8}{n^2-4} + \frac{n\sin{(n\pi/2)}+(-1)^n}{n^2-1}+(-1)^n-2\right)\,,\quad n>2\,.$$
It is easily checked that $\lambda_n <0$ for $n\ne 2$ and for all $\mu >0$. Thus, by the sign of $\lambda_2$, the uniform equilibrium is spectrally
stable for 
\begin{equation}\label{mua}
    \mu \ge \mua := \frac{f_0}{6} \,,
\end{equation}
and spectrally unstable for $\mu < \mua$.

\paragraph{Pitchfork bifurcation:}
For understanding the nature of the steady-state bifurcation at $\mu=\mua$ we observe that the problem has both a rotation symmetry 
($\vp\leftrightarrow \vp+\vp_0$ with arbitrary $\vp_0$) and a flip symmetry ($\vp\leftrightarrow -\vp$). The rotation symmetry is the reason for the double 
eigenvalues in the previous section. It can be eliminated by the additional auxiliary condition 
$$
   f(\vp=0)=f_0 \,. 
$$
This makes the eigenvalues (in particular $\lambda_2$) simple, and we can expect a branch of bifurcating solutions \cite{cran}. Because of the flip 
symmetry the generic bifurcation to be expected is a pitchfork. We shall construct a supercritical pitchfork bifurcation with a bifurcating branch for
$\mu\le\mua$, and therefore make the ansatz
	\begin{align}\label{pitchforkansatz}
	\mu = \mua -\delta^2 \,, \qquad
	f(\vp) = f_0 + \delta f_1(\vp) + \delta^2 f_2(\vp) + \delta^3 f_3(\vp) + \mathcal{O}\left(\delta^4\right) \,, 
	\end{align}
with $0<\delta\ll 1$. The corrections have to satisfy the additional auxiliary condition and mass conservation:
\begin{equation}\label{fk-aux}
    f_k(0)=0 \,,\qquad \int_{\T} f_k \md \vp = 0 \,,\qquad k\ge 1\,.
\end{equation}
Substitution of \eqref{pitchforkansatz} in the stationary version of \eqref{spathom} yields
	\begin{eqnarray}\label{deltas}
	0 &=&(L^* + Q_M)f_1 + \delta \left((L^*+Q_M)f_2 + Q(f_1,f_1) \right) \notag \\
	&& + \delta^2 \left((L^*+Q_M)f_3-\pa_{\vp}^2 f_1 + Q(f_1,f_2) + Q(f_2,f_1) \right) + \mathcal{O}(\delta^3)  \,,
	\end{eqnarray}
with $L^* := \mua \pa_\vp^2$.

\begin{lemma}
The null space of $L^*+Q_M$ subject to (\ref{fk-aux}) is one-dimensional and spanned by $\sin(2\vp)$. The solvability condition for the 
equation $(L^*+Q_M)f = g$ is
$$
    \int_{\T} g\,\sin(2\vp)\md\vp = 0 \,.
$$
\end{lemma}

\begin{proof}
The result on the null space is a consequence of the computations in the previous section and of the observation that the $\cos(2\vp)$-contribution
is eliminated by the additional auxiliary condition. A straightforward computation shows that $Q_M$ is symmetric with respect to the $L^2$ scalar product,
and so is of course $L^*$, completing the proof.
\end{proof}

Equation (\ref{deltas}) with $\delta=0$ implies
$$
   f_1(\vp) = b \sin(2\vp) \,,
$$
with $b\in\R$ still to be determined. By a straightforward computation the inhomogeneity in the $\mathcal{O}(\delta)$-equation is given by
$$
   Q(f_1,f_1) = -\frac{4}{3}b^2 \cos(4\vp) \,,
$$
which satisfies the solvability condition for
$$
   (L^*+Q_M)f_2 + Q(f_1,f_1) = 0 \,.
$$
The computations in the previous section show that $(L^*+Q_M)\cos(4\vp) = \lambda_4^* \cos(4\vp)$ with
$$
    \lambda_4^* = \lambda_4 \bigm|_{\mu=\mua} = -\frac{88}{15}f_0 \,.
$$
Therefore, considering (\ref{fk-aux}), we obtain
$$
   f_2(\vp) = \frac{5b^2}{22 f_0} (\cos(2\vp)-\cos(4\vp)) \,.
$$
The final computation is the evaluation of the solvability condition 
$$
    \int_{\T} \left(-\pa_{\vp}^2 f_1 + Q(f_1,f_2) + Q(f_2,f_1)\right)\sin(2\vp)\md\vp = 0
$$
for the $\mathcal{O}(\delta^2)$-equation, which gives
$$
     b\left(\frac{440 f_0}{15} - b^2\right) = 0 \,.
$$
The nontrivial solutions determine the bifurcating branch
$$
   f(\vp) = f_0 \pm \sqrt{\frac{440 f_0(\mua-\mu)}{15}} \sin(2\vp) + \mathcal{O}(\mua-\mu) \,,\qquad \mu\le\mua \,,
$$
of nontrivial equilibria. Supercriticality implies stability of the bifurcating branch (see the bifurcation diagram in Figure \ref{pitchfork}).
Recalling mass conservation and rotational symmetry we obtain for each $0<\mua-\mu\ll 1$ a two-dimensional set of nontrivial equilibria
of the form
$$
   f(\vp) = \frac{M}{2\pi} + \sqrt{\frac{220 M(\mua-\mu)}{15\pi}} \sin(2(\vp-\vp_0)) + \mathcal{O}(\mua-\mu) \,,
$$
parametrized by the mass $M\ge 0$ and the rotation angle $\vp_0\in [0,\pi)$. Note that the different signs in the second term can be realized by rotation 
by $\pi/2$.
	
\begin{figure}[H]
	\centering
	\includegraphics[width=0.5\textwidth]{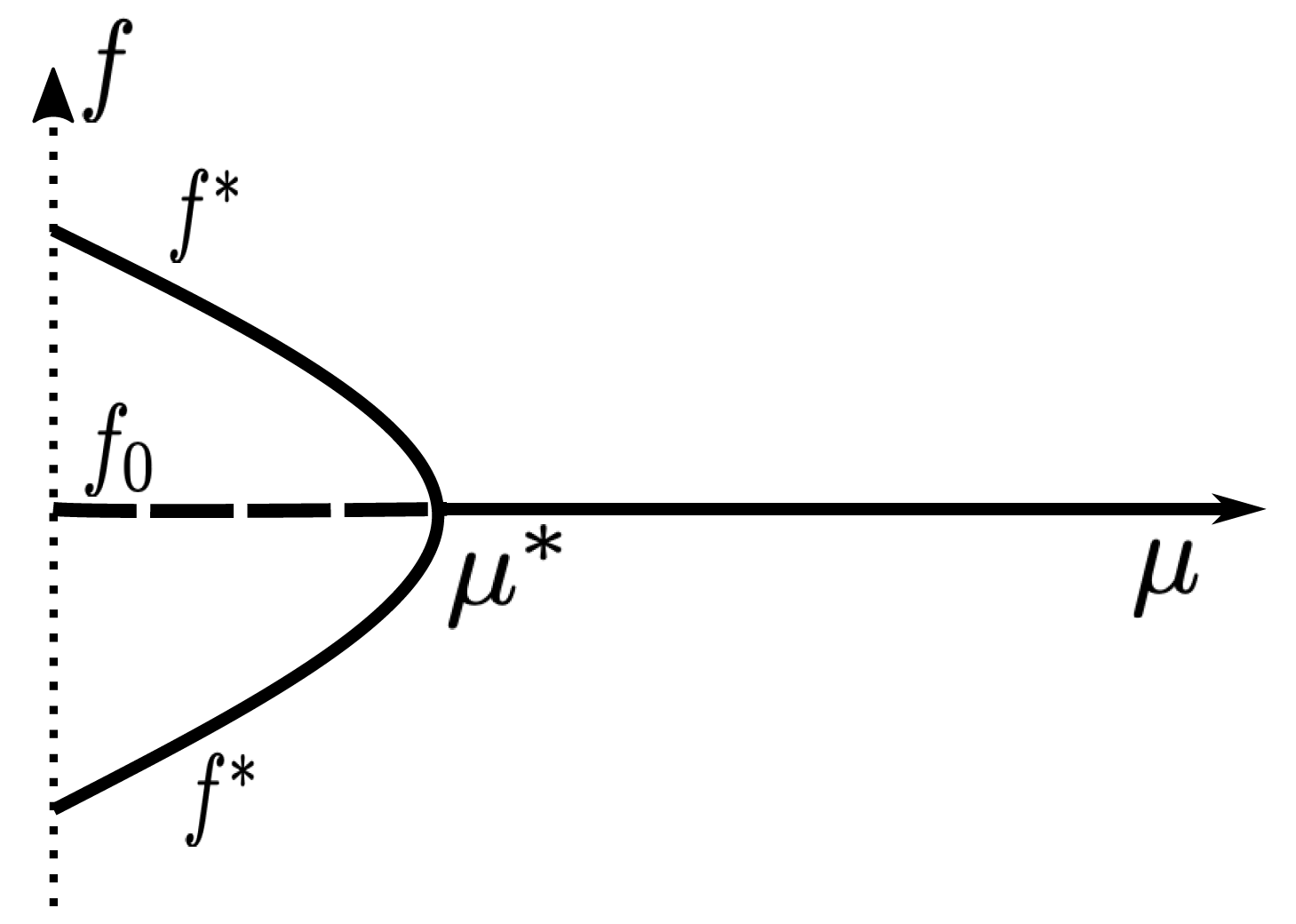}
	\caption{Bifurcation diagram of the supercritical pitchfork bifurcation. Solid lines represent stable branches, 
	      while the dashed line represents the unstable region of the uniform equilibrium.}
	\label{pitchfork}
\end{figure}

\begin{rem}
The above calculations have been carried out for $b(\vp,\vpa)=|\sin(\vp-\vpa)|$, modelling the case of rod-shaped myxobacteria. 
It is easily checked that the important property that the linearized operator is diagonalized by Fourier decomposition holds for every model
of the physically reasonable form $b(\vp,\vpa)=\hat b(|\vp-\vpa|)$.
In particular, the same type of bifurcation result holds for Maxwellian myxos, i.e. $b \equiv 1$, with the bifurcation value 
$$
   \mua_{Maxwell} = f_0\left(1-\frac{\pi}{4}\right) > \mua \,.
$$
\end{rem}

\subsection{Equilibria for Small Diffusivity}\label{ex}
	
In \cite{ourpaper} it has been shown that the set of nontrivial equilibria of $Q$ is three-dimensional and of the form
\begin{equation}\label{f-infty}
   f_\infty(\vp) :=\rho_+ \delta(\vp-\vp_0) +\rho_-\delta(\vp-\vp_0^\downarrow) \,,
 \end{equation}
 with arbitrary $\vp_0\in\T$ and $\rho_+,\rho_-\ge 0$, satisfying $\rho_+ + \rho_- = M$. On the other hand, in the previous section we have found a manifold
 of equilibria of $\mu\pa_\vp^2 + Q$, which is two-dimensional for each $\mu$ smaller than and close to $\mu^*$. The question is: Can these results be
 connected by the limit $\mu\to 0+$? Motivated by the fact that the bifurcating equilibria of the previous section have two symmetric maxima, our conjecture
 is the following: {\em The manifold of nontrivial equilibria starting at the bifurcation at $\mu=\mua$ can be extended to arbitrarily small $\mu>0$. Its limit as
 $\mu\to 0+$ is the family}
 \begin{equation}\label{f-infty-sym}
   f_\infty(\vp) :=\frac{M}{2} \delta(\vp-\vp_0) + \frac{M}{2}\delta(\vp-\vp_0^\downarrow) \,,\qquad M\ge 0 \,,\quad \vp_0\in\T \,.
 \end{equation}
 {\em For small $\mu$, $\mu\pa_\vp^2 + Q$ possesses a three-dimensional family of metastable states close to (\ref{f-infty}).}
 
So far we cannot prove any of this, but in the remainder of this section we shall present a first small step: We shall prove the existence of a formal 
approximation for equilibria of $\mu\pa_\vp^2 + Q$ close to (\ref{f-infty-sym}) for the case of Maxwellian myxos. Some more evidence will be provided 
by the numerical simulations presented in the following section.
 
The stationary equation with $b(\vp,\vpa) \equiv 1$ can be written as
\begin{align}\label{stamax}
	0 =\mu \pa_{\vp}^2f(\vp) + 2 \int_{\vp-\frac{\pi}{4}}^{\vp+\frac{\pi}{4}} f(2\vp-\vpa) f(\vpa) \; \md \vpa 
	+ f(\vp+\pi) \int_{\vp-\frac{\pi}{2}}^{\vp+\frac{\pi}{2}} f(\vpa) \; \md \vpa -M f(\vp).
\end{align}
We look for a reflection symmetric solution ($f(\vp) = f(\vp + \pi)$) with mass concentrated around $\vp=0$ and $\vp=\pi$ (close to (\ref{f-infty-sym}) with
$\vp_0=0$). Concentrating on the peak at $\vp=0$, we introduce the new unknown $F(\xi)$ by the scaling
$$
    \vp = \sqrt{\frac{2\mu}{M}} \,\xi \,,\qquad f = \frac{M}{2}\sqrt{\frac{M}{2\mu}} \,F \,,
$$
and rewrite (\ref{stamax}) as
$$
   0 =  \pa_\xi^2 F(\xi) + 2 \int_{\xi-\frac{\pi}{4}\sqrt{\frac{M}{2\mu}}}^{\xi+\frac{\pi}{4}\sqrt{\frac{M}{2\mu}}} F(2\xi-\xia)F(\xia)\md\xia
    + F(\xi) \left( \frac{2}{M}\int_{\xi\sqrt{\frac{2\mu}{M}} - \frac{\pi}{2}}^{\xi\sqrt{\frac{2\mu}{M}} + \frac{\pi}{2}} f(\vpa)\md\vpa - 2\right)
$$
Since, by the symmetry assumption, 
$$
    \int_{-\frac{\pi}{2}\sqrt{\frac{M}{2\mu}}}^{\frac{\pi}{2}\sqrt{\frac{M}{2\mu}}} F(\xia)\md\xia = \frac{2}{M} \int_{-\pi/2}^{\pi/2} f(\vpa)\md\vpa = 1 \,,
$$
holds, the limit $\mu\to 0$ gives
\begin{equation}\label{F-equ}
   0 =  \pa_\xi^2 F(\xi) + 2 \int_{-\infty}^\infty F(2\xi-\xia)F(\xia)\md\xia - F(\xi) \,,\qquad \int_{-\infty}^\infty F(\xi)\md\xi = 1 \,.
\end{equation}
With the Green's function of $\pa_\xi^2 -$id we can rewrite \eqref{F-equ} as the fixed point problem
$$
   F = \mathcal{S}(F) \qquad\mbox{with}\quad
   \mathcal{S}(F)(\xi) := \int_{-\infty}^\infty \int_{-\infty}^\infty e^{-|\xi-\tilde\xi|} F(2\tilde\xi -\xia)F(\xia)\md\xia \md\tilde{\xi} \,. 
$$
We claim that $\mathcal{S}$ maps the set
\begin{align*}
    \mathcal{B} := \left\{F \in L_+^1(\R)\cap C_B(\R):\, F(\xi)=F(-\xi) \,,\,\int_{-\infty}^\infty F(\xi)\md\xi = 1 \,,\, \int_{-\infty}^\infty \xi^2 F(\xi) \md \xi = 4 \right\} \,.
\end{align*}
into itself. For $F\in \mathcal{B}$ we obviously have $\mathcal{S}(F)\ge 0$, $\mathcal{S}(F)(\xi) = \mathcal{S}(F)(-\xi)$,  and
$$
   \int_{-\infty}^\infty \mathcal{S}(F)(\xi)\md\xi = 2\int_{-\infty}^\infty \int_{-\infty}^\infty F(2\tilde\xi -\xia)F(\xia)\md\xia \md\tilde{\xi} 
   = \int_{-\infty}^\infty \int_{-\infty}^\infty F(\hat\xi)F(\xia)\md\xia \md\hat\xi = 1  \,.
$$
A preliminary computation for the evaluation of the variance is
$$
   \int_{-\infty}^\infty \xi^2 e^{-|\xi-\tilde\xi|} \md\xi = 4 + 2\tilde\xi^2 \,,
$$
implying
\begin{eqnarray*}
   \int_{-\infty}^\infty \xi^2 \mathcal{S}(F)(\xi) \md\xi &=& 4\int_{-\infty}^\infty \int_{-\infty}^\infty F(2\tilde\xi -\xia)F(\xia)\md\xia \md\tilde{\xi} + 
     2 \int_{-\infty}^\infty \int_{-\infty}^\infty \tilde\xi^2 F(2\tilde\xi -\xia)F(\xia)\md\xia \md\tilde{\xi} \\
   &=& 2 + \frac{1}{4} \int_{-\infty}^\infty \int_{-\infty}^\infty (\xia + \hat\xi)^2 F(\hat\xi)F(\xia) \md\xia \md\hat\xi \\
   &=& 2 + \frac{1}{2} \int_{-\infty}^\infty \int_{-\infty}^\infty (\xia^2 + \xia\hat\xi) F(\hat\xi)F(\xia) \md\xia \md\hat\xi = 4 \,,
\end{eqnarray*}
where the evenness of $F$ has been used in the last equality. Finally, for any $\xi\in\R$,
$$
    \mathcal{S}(F)(\xi) \le \int_{-\infty}^\infty \int_{-\infty}^\infty F(2\tilde\xi -\xia)F(\xia)\md\xia \md\tilde{\xi} = \frac{1}{2}\,,
$$
whence $\mathcal{S}:\, \mathcal{B}\to\mathcal{B}$, since the uniform continuity of $\mathcal{S}(F)$ is obvious. 

\begin{lemma}
With the above definitions, the set $\mathcal{S}(\mathcal{B})$ is relatively compact in $C_B(\R)$.
\end{lemma}

\begin{proof}
By the estimate
$$
   |\mathcal{S}(F)(\xi_1) - \mathcal{S}(F)(\xi_2)| \le \int_{-\infty}^\infty \int_{-\infty}^\infty |\xi_1 - \xi_2| F(2\tilde\xi -\xia)F(\xia)\md\xia \md\tilde{\xi} 
   = \frac{1}{2} |\xi_1 - \xi_2| \,,
$$
$\mathcal{S}(\mathcal{B})$ is equi-Lipschitz-continuous. By the Arzel\'a-Ascoli theorem, a sequence $\{F_n\}\subset \mathcal{S}(\mathcal{B})$ possesses 
for every compact set $K\subset\R$ a subsequence, which converges uniformly on $K$. The standard diagonal procedure produces one subsequence
$\{G_n\}\subset \{F_n\}$, such that $G_n\to G$ pointwise in $\R$ and uniformly on each compact set. It remains to prove that the convergence is uniform 
on $\R$. 

Let $\xi_0\ge X>0$. Then, as a consequence of the Lipschitz continuity,
$$
    F(\xi) \ge F(\xi_0) - \frac{1}{2}(\xi-\xi_0) \,,\qquad\mbox{for } \xi_0 \le \xi \le \xi_0 + 2F(\xi_0) \,.
$$
Thus,
$$
    F(\xi_0)^2 = \int_{\xi_0}^{\xi_0 + 2F(\xi_0)} \left( F(\xi_0) - \frac{1}{2}(\xi-\xi_0)\right)\md\xi \le \int_X^\infty F(\xi)\md\xi
    \le \int_0^\infty \frac{\xi^2}{X^2} F(\xi)\md\xi = \frac{2}{X^2}
$$
With the analogous estimate for $\xi_0 \le -X$ we have
$$
     F(\xi) \le \frac{\sqrt{2}}{X} \,,\qquad\mbox{for}\quad F\in\mathcal{S}(\mathcal{B}) \,,\quad |\xi| \ge X \,.
$$
The same is true for the pointwise limit $G$ of $\{G_n\}$ and, thus,
$$
   \sup_{\R}|G_n-G| = \max\left\{ \frac{2\sqrt{2}}{X}\,,\, \sup_{(-X,X)} |G_n-G| \right\}\,,
$$
which can be made arbitrarily small by choosing first $X$ and then $n$ sufficiently large.
\end{proof}

\begin{theorem}
Problem \eqref{F-equ} has a nonnegative smooth solution satisfying
$$
   \int_{-\infty}^\infty \xi^2 F(\xi) \md \xi \leq 4 \,.
$$
\end{theorem}

\begin{proof}
For an application of the Schauder fixed point theorem it remains to prove continuity of $\mathcal{S}$ with respect to the supremum norm: For $F_1,F_2\in\mathcal{B}$,
\begin{eqnarray*}
   |\mathcal{S}(F_1)(\xi) - \mathcal{S}(F_2)(\xi)| &\le&  \int_{-\infty}^\infty \int_{-\infty}^\infty e^{-|\xi-\tilde\xi|} 
  \Bigl( \left|F_1(2\tilde\xi -\xia)-F_2(2\tilde\xi-\xia)\right|F_1(\xia) \\
  && \hskip 3cm + \left|F_1(\xia)- F_2(\xia)\right| F_2(2\tilde\xi-\xia)\Bigr)\md\xia \md\tilde{\xi} \\
  &\le& \sup_{\R}|F_1-F_2| \int_{-\infty}^\infty \int_{-\infty}^\infty e^{-|\xi-\tilde\xi|} \left( F_1(\xia) + F_2(2\tilde\xi-\xia)\right)\md\xia \md\tilde\xi  \\
   &=& 4\sup_{\R}|F_1-F_2| \,.
\end{eqnarray*}
An application of the Schauder theorem shows that $\mathcal{S}$ has a Lipschitz continuous fixed point. The boundedness of the variance implies 
tightness and therefore the fixed point is a probability density, satisfying the upper bound of the variance. It is easily seen that the map
$\xi\mapsto \int_{-\infty}^\infty F(2\xi-\xia)F(\xia)\md\xia$ is Lipschitz continuous and therefore the differential equation \eqref{F-equ} implies $F\in C^2(\R)$.
Bootstrapping gives higher regularity.
\end{proof}

A formal approximation as $\mu\to 0+$ for an equilibrium can now be given as
$$
     f(\vp) \approx \left\{ \begin{array}{ll}  \frac{M}{2}\sqrt{\frac{M}{2\mu}} \,F\left(\vp\sqrt{\frac{M}{2\mu}}\right) \,, & |\vp| \le \frac{\pi}{2} \,,\\
     \frac{M}{2}\sqrt{\frac{M}{2\mu}} \,F\left((\vp-\pi)\sqrt{\frac{M}{2\mu}}\right) \,, & |\vp-\pi| \le \frac{\pi}{2}\,. \end{array} \right.     
$$
The rigorous justification remains open.

\section{Numerical Simulations with the Spatially Homogeneous Model}\label{sec:num}
	
\paragraph{Discretization:}
The results of the preceding section will be illustrated by numerical simulations with the spatially homogeneous model (\ref{spathom}). Discretization in the angular direction is based on an equidistant grid
	$$
	\vp_k = \frac{k\pi}{n} \,,\qquad k \in \mathbb{Z}_{2n} \,,
	$$
with an even number of grid points, representing a discrete torus. The collision operator is approximated by quadrature, chosen such that 
mass is conserved and post-collisional states are on the grid (see \cite[Section 5]{ourpaper} for details). Diffusion is discretized by the standard three-point scheme,
and the explicit Euler scheme is used for the time discretization, such that mass conservation is guaranteed.	
The scheme has been implemented in \textsc{Matlab}.

All simulations have been carried out with $n=51$ and with time steps satisfying a parabolic CFL condition. This has not been too restrictive since only
rather small values for the diffusivity $\mu$ have been used. The mass has been normalized, i.e. $M=1$, leading to the bifurcation value (see \eqref{mua})
$$
    \mua= \frac{1}{12\pi} \approx 0.0265 \,.
$$ 
The plots in the figures below show distributions initially (red dotted lines), at an intermediate time (blue dashed lines), and at the end of the simulation time
(black solid lines), the latter typically close to an equilibrium state. 
	
\paragraph{Simulations in the bifurcation regime:}
First we show simulations with values of the diffusivity $\mu$ just below and just above the bifurcation value $\mua$.

In Figure \ref{bif1} the initial data have been chosen as random perturbations of the constant equilibrium, which is stable for $\mu>\mua$ (right), and
unstable for $\mu<\mua$ (left).  In the latter case, the solution converges to a nonuniform steady state with peaks centered around two unpredictable, but always opposite points. 

In Figure \ref{bif2} the constant equilibrium is initially perturbed only at one grid point. The results are as above, except that the nonuniform steady state
is not quite reached at the end of the simulation time, since the growth of the small initial perturbation takes much longer than in the first experiment.

\begin{figure}[H]
		\centering
		\begin{subfigure}{0.48\textwidth} 
			\includegraphics[width=\textwidth]{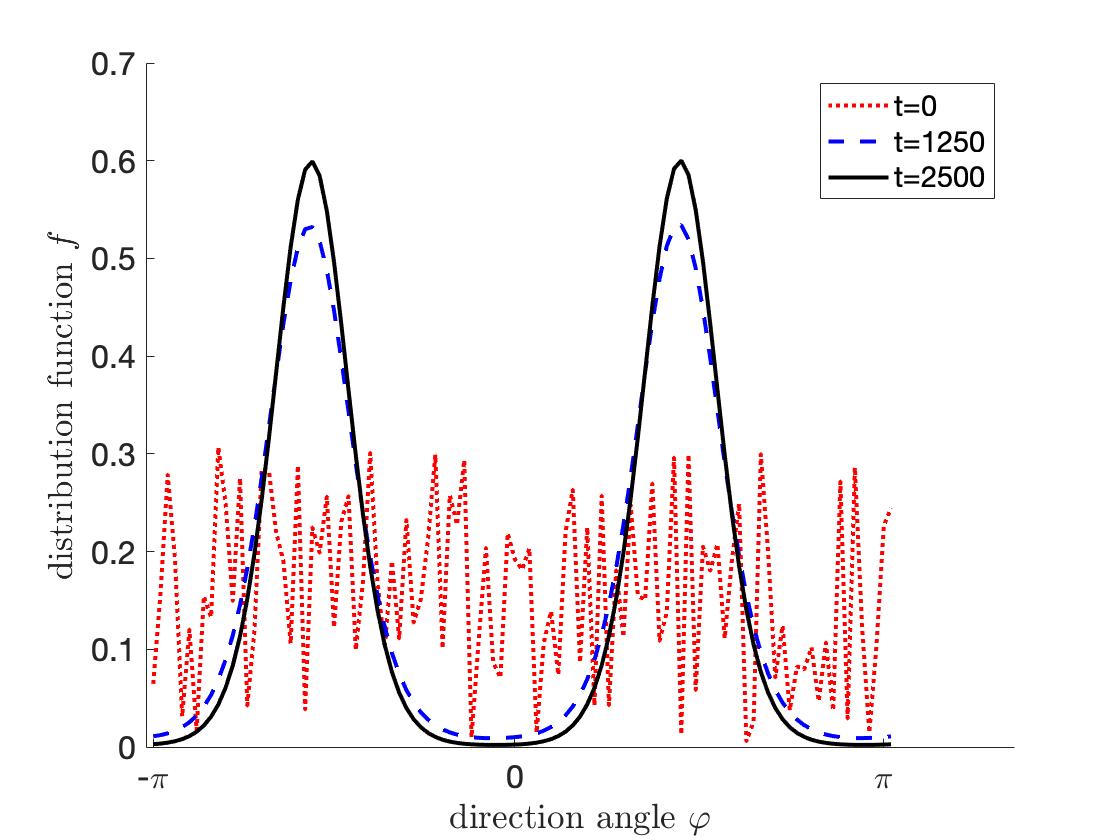}	
		\end{subfigure}
		\hspace{1em} 
		\begin{subfigure}{0.48\textwidth} 
			\includegraphics[width=\textwidth]{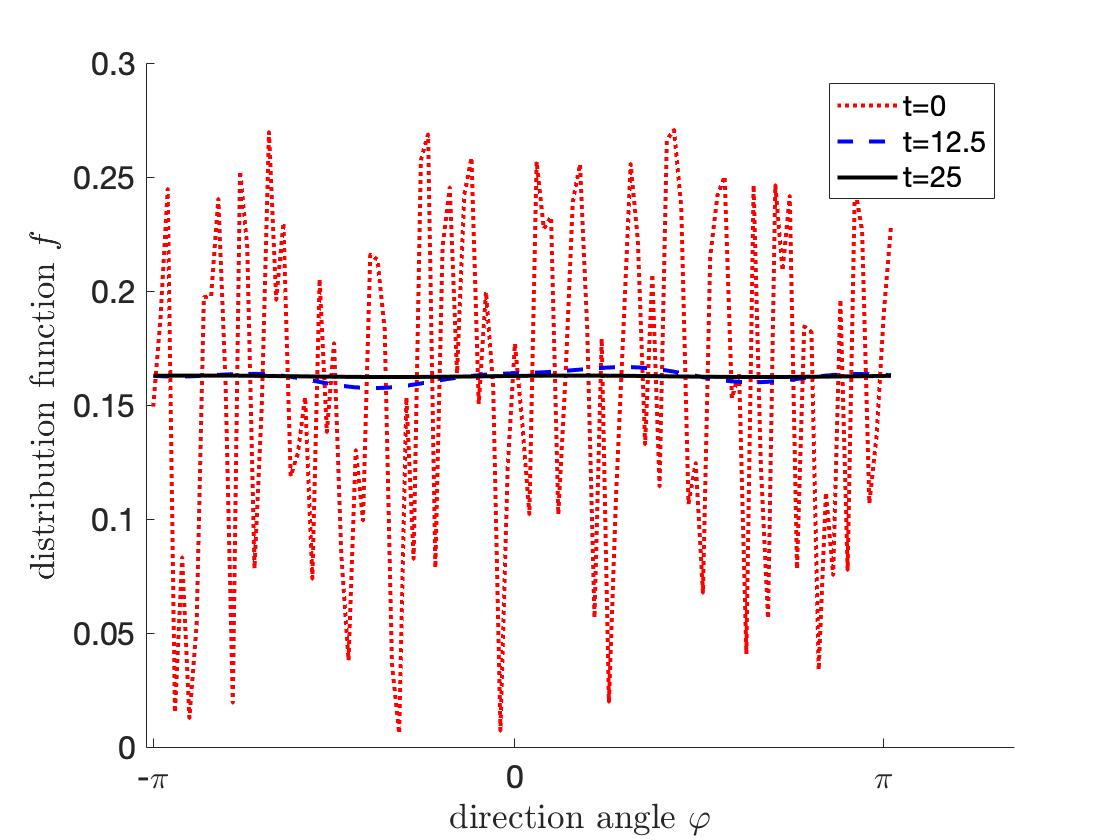}
		\end{subfigure}
		
\caption{Random perturbation of the constant equilibrium as initial conditions. \emph{Left:} With diffusivity smaller than the bifurcation value ($\mu = 0.02$) 
the solution converges to a nonuniform equilibrium with peaks at unpredictable positions. \emph{Right:} For $\mu = 0.03>\mua$ convergence to
the constant steady state $f_0=1/(2\pi)$ is observed.}
		\label{bif1}
\end{figure}
	
\begin{figure}[H]
		\centering
		\begin{subfigure}{0.48\textwidth} 
			\includegraphics[width=\textwidth]{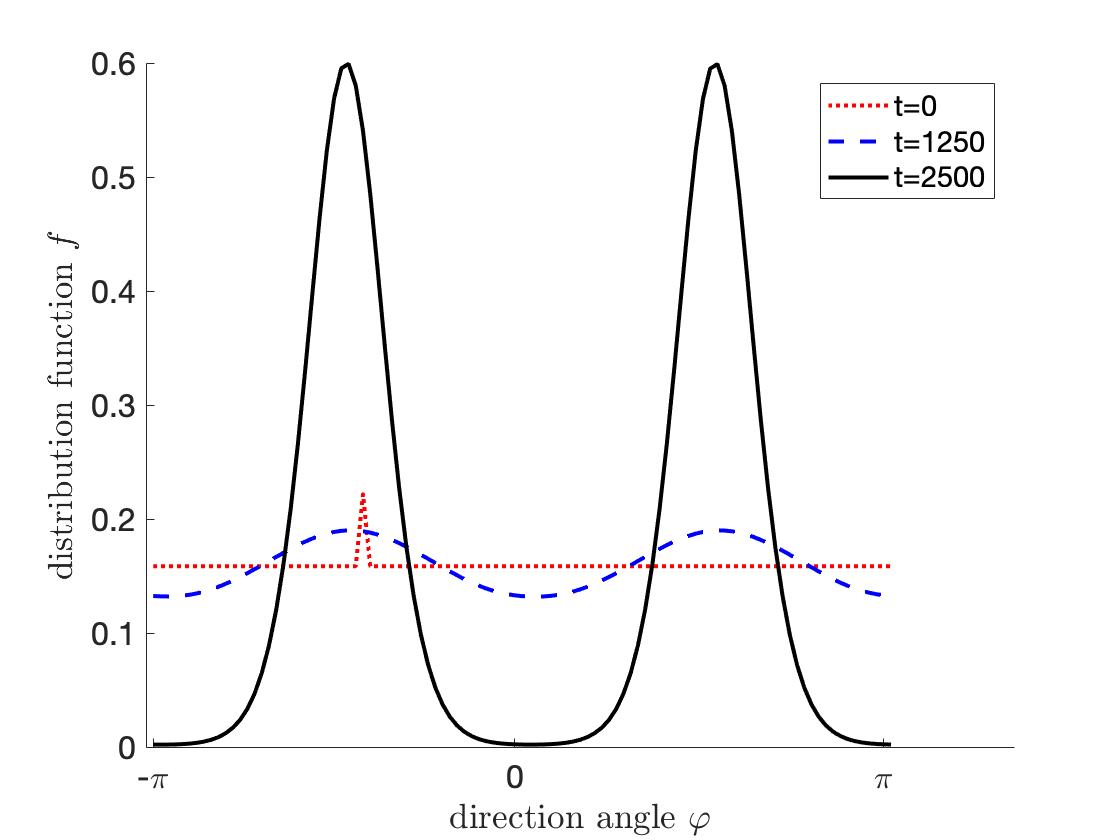}	
		\end{subfigure}
		\hspace{1em} 
		\begin{subfigure}{0.48\textwidth} 
			\includegraphics[width=\textwidth]{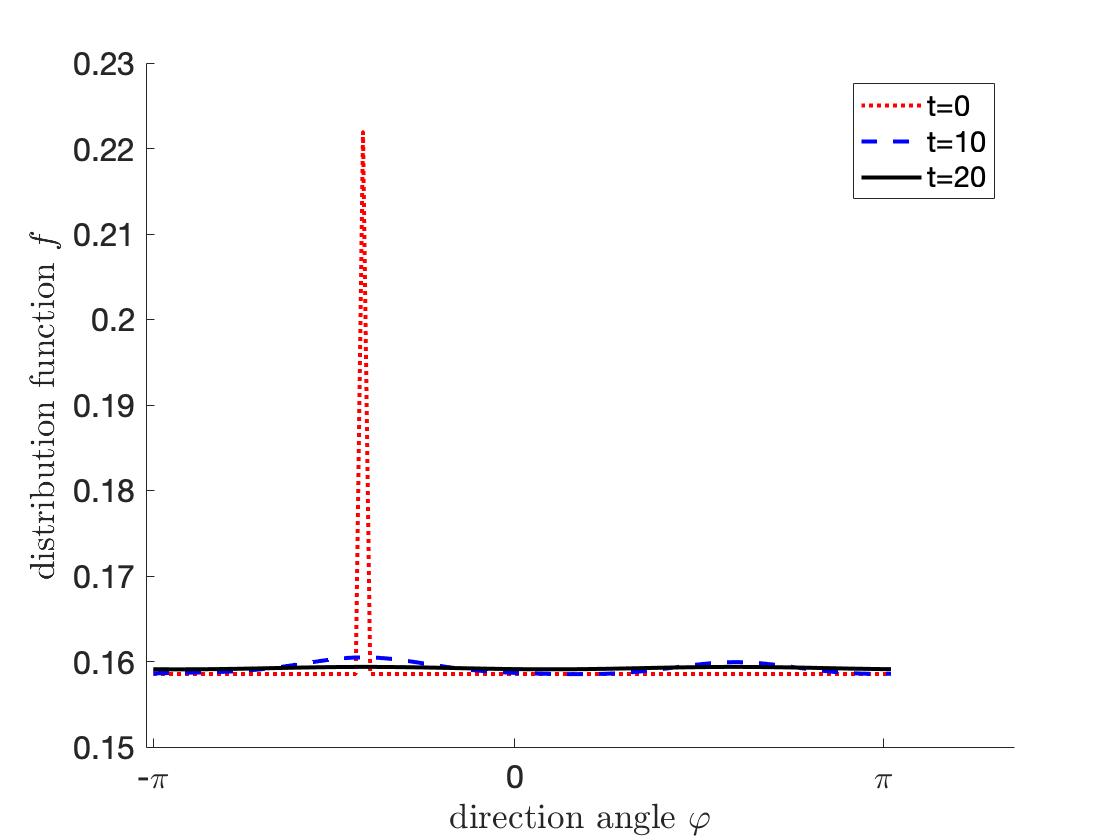}
		\end{subfigure}
		
\caption{Initial perturbation of the constant equilibrium at one grid point. \emph{Left:}  $\mu = 0.02<\mua$. \emph{Right:} $\mu = 0.03>\mua$.}
		\label{bif2}
\end{figure}

\paragraph{Simulations in the small diffusion regime:}
The remaining simulation results support the conjecture formulated in Section \ref{ex}. For the value $\mu=0.001$ of the diffusivity we always observe 
convergence to a nonuniform equilibrium with opposite peaks of equal mass. As expected the dynamics passes through metastable states with two
peaks of different masses, where the convergence to the final equilibrium becomes slower with decreasing values of $\mu$. This is the reason why a
rather moderate value has been chosen, where the concentration effect is not too strong.

The simulation shown in Figures \ref{small1} starts with two opposite plateaus of different mass, which are smoothed rather fast, before mass is
transferred by diffusion and the symmetrizing effect of the reversal operator to produce peaks of equal size. In Figure \ref{small3} the initial datum is nonzero only at two non-opposite points with different values.
In this case not only mass has to be transferred, but the peaks also move to produce the distance $\pi$ between them. Both figures also show the masses
in opposite half intervals converging towards each other. For $\mu>0$ we never observe unsymmetric equilibria, which exist and are stable for $\mu=0$
\cite{ourpaper}.

	\begin{figure}[H]
		\centering
		\hspace{1em} 
		\begin{subfigure}{0.48\textwidth} 
			\includegraphics[width=\textwidth]{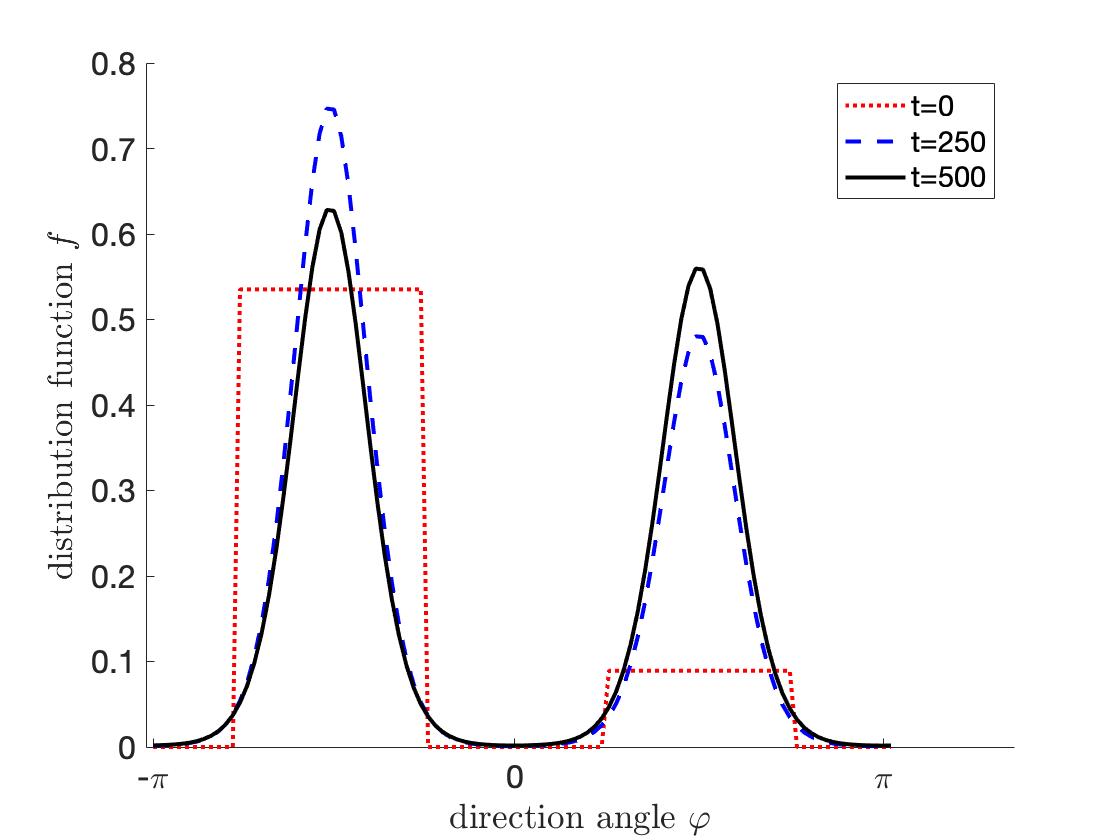}
		\end{subfigure}
		\begin{subfigure}{0.48\textwidth} 
			\includegraphics[width=\textwidth]{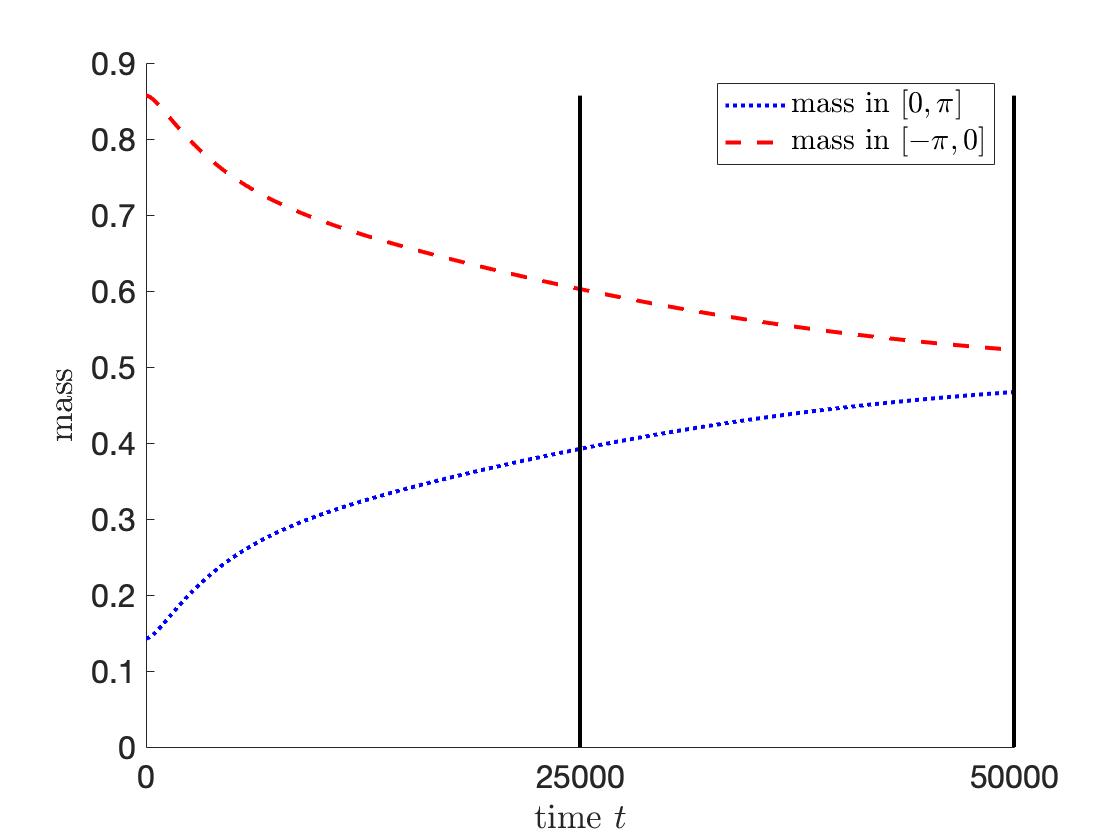}	
		\end{subfigure}
		
		\caption{Initial condition with equally distributed different masses in  $[-3\pi/4,-\pi/4]$ and $[\pi/4,3\pi/4]$, diffusion constant $\mu=0.001$. 
			\emph{Left:} Time evolution with smoothing of the plateaus, followed by redistribution of mass. \emph{Right:} Time evolution of the mass
			in $[-\pi,0]$ and in $[0,\pi]$.}
		\label{small1}
	\end{figure}
	
	%
	
	\begin{figure}[H]
		\centering
		\begin{subfigure}{0.48\textwidth} 
			\includegraphics[width=\textwidth]{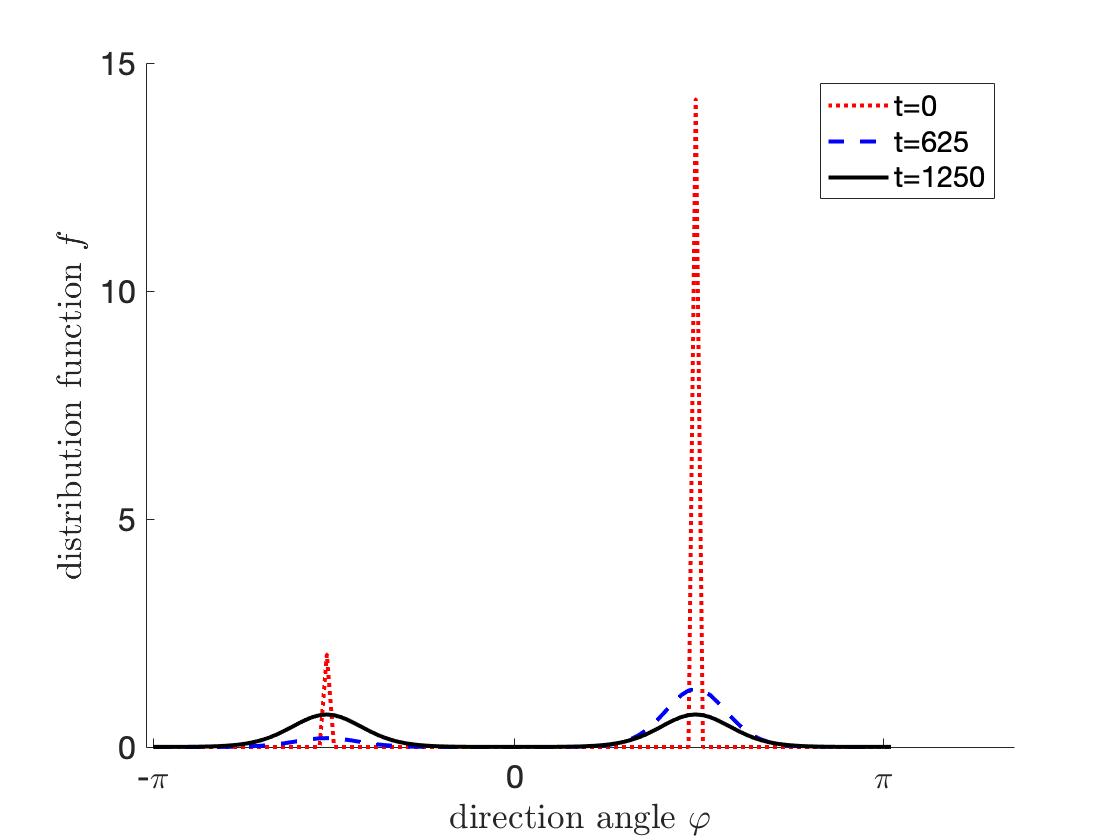}
		\end{subfigure}
		\begin{subfigure}{0.48\textwidth} 
			\includegraphics[width=\textwidth]{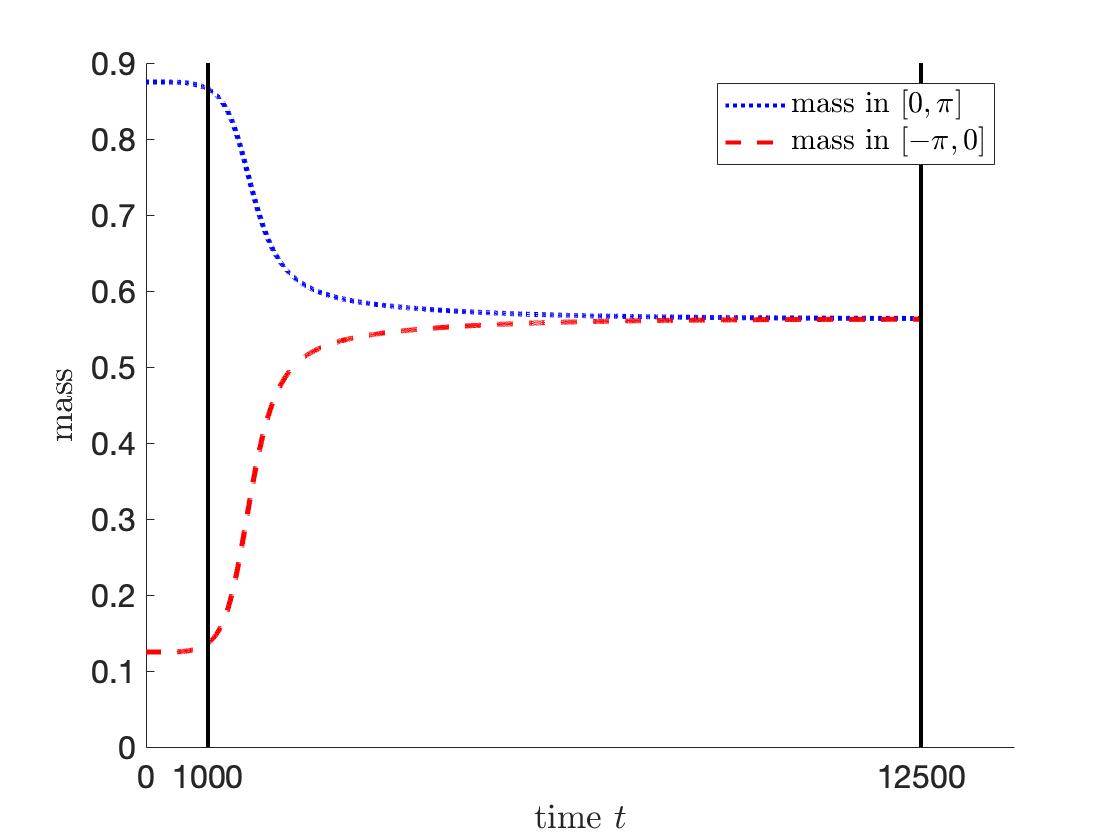}	
		\end{subfigure}
		
		\caption{Different masses initially concentrated at $-3\pi/4$ and at $\pi/2$, diffusion constant $\mu=0.001$. 
			\emph{Left:} Time evolution with smoothing and relocation of the peaks, followed by redistribution of mass. \emph{Right:} Time evolution 
			of the mass in $[-\pi,0]$ and in $[0,\pi]$.}
		\label{small3}
	\end{figure}

\end{document}